\documentclass[a4paper,10pt]{article}

\usepackage{geometry}
\usepackage{subcaption}
\usepackage{verbatim}
\usepackage{caption}
\usepackage{url}
\usepackage{amsmath}
\usepackage{geometry}
\usepackage{amssymb}
\usepackage{amsmath}
\usepackage{graphicx}
\usepackage{amsthm}
\usepackage{bbm}
\usepackage{float}
\usepackage{color,soul}
\usepackage[hidelinks]{hyperref}
\usepackage[T1]{fontenc}
\usepackage[utf8]{inputenc}
\usepackage{authblk}
\usepackage{booktabs}
\usepackage{longtable}
\usepackage{enumerate}
\usepackage{tikz}
\usetikzlibrary{arrows}
\usetikzlibrary{decorations.markings}
\usepackage{indentfirst}
\newcommand{\eChar}{\begin{enumerate}[(i)]}
\newcommand{\eCharR}{\begin{enumerate}[(a)]}
\newcommand{\eBr}{\begin{enumerate}[(1)]}

\newcommand{\eps}{\varepsilon}

\title{On a maximal anti-Ramsey conjecture of Burr, Erd\H{o}s, Graham, and S\'os}

\author{ }

\date{ }

\theoremstyle{plain}
\newtheorem{lemma}{Lemma}[section]
\newtheorem{theorem}[lemma]{Theorem}

\theoremstyle{definition}

\newtheorem{claim}{Claim}

\newtheorem{conjecture}[lemma]{Conjecture}

\newtheorem*{defn*}{Definition}

\newcommand{\di}{\operatorname{dist}}

\newcommand{\ex}{\operatorname{ex}}

\begin{document}

\author{
Matija Buci\'c\thanks{Institute of Mathematics, University of Vienna, Vienna, Austria, and Department of Mathematics, Princeton University, Princeton, America. Research supported by National Science Foundation grant DMS-2349013. Email: {\tt mb5225@princeton.edu}.}
~~~~
Kaizhe Chen\thanks{School of the Gifted Young, University of Science and Technology of China, Hefei, Anhui 230026, China. Research supported by National Natural Science Foundation of China grant 125B1009. Email: {\tt ckz22000259@mail.ustc.edu.cn}.}
~~~~
Jie Ma\thanks{School of Mathematical Sciences, University of Science and Technology of China, Hefei, Anhui 230026, and Yau Mathematical Sciences Center, Tsinghua University, Beijing 100084, China.
Research supported by National Key Research and Development Program of China 2023YFA1010201 and National Natural Science Foundation of China grant 12125106. Email: {\tt jiema@ustc.edu.cn}.}
}

\maketitle

\begin{abstract}
    Given a graph $H$, the maximal anti-Ramsey function $f(n,e,H)$ denotes the minimum integer $f$ for which there exists an $n$-vertex graph $G$ with at least $e$ edges admitting an edge-coloring with $f$ colors in which each copy of $H$ in $G$ is rainbow.
    In the late 1980s, Burr, Erd\H{o}s, Graham, and S\'os conjectured that for every odd cycle $C_{2k+1}$ with $k \ge 3$,
    $f(n, \lfloor n^2/4 \rfloor + 1, C_{2k+1}) = n^2/8 + o(n^2)$.
    In this note, we confirm this conjecture for all $k \ge 4$. More generally, we establish the asymptotic formula $$f\left(n,e,C_{2k+1}\right)=\frac{e}{2}+\frac{n}{2}\sqrt{e-\frac{n^2}{4}}+o(n^2),$$
    for the entire non-trivial range of $\left\lfloor n^2/4 \right\rfloor+1\le e\le \binom{n}{2}$.
\end{abstract}

\section{Introduction}
\noindent Ramsey theory refers to a large body of results, all roughly speaking, saying that any sufficiently large structure, no matter how chaotic, needs to have a well-organised substructure. In the original graph theoretic instance, which underlies many surprising applications beyond, the large structures being considered are graphs, and one encodes the structure by coloring the edges. In this setting, a well-organised substructure is usually taken to mean a monochromatic one. A first result of this flavour was proved by Ramsey in 1929 \cite{ramsey}, and Ramsey theory has since grown into a fully fledged, extremely active area of research.

In the 1970s, Erd\H{o}s, Simonovits, and S\'os \cite{ESS} initiated the study of so-called anti-Ramsey theory in which one, instead of a well-organised substructure, seeks a very chaotic one. Here, in the original graph theoretic instance, a chaotic structure is captured by the notion of rainbow coloring. Namely, in an edge-coloring of a graph $G$, a subgraph $H$ of $G$ is said to be {\it rainbow} if all its edges have distinct colors. A typical question in anti-Ramsey theory asks, for a fixed graph \(H\) and a graph \(G\) (such as the clique \(K_n\)), to find the minimum \(r\) such that any \(r\)-edge-coloring of \(G\) contains a rainbow copy of \(H\) (see, e.g., \cite{A,AJ,ESS,JW}).

In the 1980s, Burr, Erd\H{o}s, Graham, and S\'os \cite{BEGS} 
initiated the study of a certain dual version of this classical anti-Ramsey problem. 
Namely, they ask, for a fixed graph $H$ and a graph $G$, to find the minimum $r$ such that there exists an $r$-edge-coloring of $G$ in which \emph{every} copy of $H$ is rainbow.
Namely, when is it possible to have a completely locally chaotic coloring? 
The contrapositive here frames this as asking in some sense the weakest possible Ramsey question, where we are seeking to guarantee at least one non-rainbow copy of a given target graph.

The following maximal anti-Ramsey function, 
was introduced in \cite{BEFGS,BEGS}, reiterated many times, including in the famous book Erd\H{o}s on Graphs \cite{erd-graphs} by Chung and Graham, and further studied in \cite{KS,sarkozy-selkow,S,SS}, to name just a few examples.

\begin{defn*}
    Given a graph $H$ and $n,e>0,$ we denote by $f(n,e,H)$ the minimum integer $f$ for which there exists an $n$-vertex graph $G$ with at least $e$ edges admitting an edge-coloring with $f$ colors in which each copy of $H$ in $G$ is rainbow.

\noindent In other words, what is the largest $f$ such that any $f$-edge-colored graph on $n$ vertices with $e$ edges contains a non-rainbow copy of $H$?
\end{defn*}

\noindent The original motivation for Burr, Erd\H{o}s, Graham, and S\'os came from a question of Berkowitz in computer science, related to time/space tradeoffs for Turing machines (see \cite{BEFGS,BEGS}). Various instances of the question have close ties to the infamous problem from additive combinatorics of determining the largest size of a subset of $\{1,2,\ldots, n\}$ without a $k$-term arithmetic progression and the infamous Brown-Erd\H{o}s-S\'os Conjecture (see \cite{BEGS} for more details).

The maximal anti-Ramsey function $f(n,e,H)$ has an immediate connection to extremal graph theory. Here, one of the central objects of study is the Tur\'an number of a graph. The $n$-vertex \emph{Tur\'an number} of $H$, denoted $\ex(n,H)$, is defined as the maximum number of edges in an $n$-vertex graph that contains no copy of $H$.
Assuming $H$ has at least two edges, it is immediate that $f(n,e,H)=1$ if and only if $e\le \ex(n,H)$. This implies that the question of determining the behaviour of the maximal anti-Ramsey function $f(n,e,H)$ is only interesting if $e> \ex(n,H)$. Perhaps the most immediate ``stability'' question is what happens if $e=\ex(n,H)+1$. This already turns out to be very interesting and was first explored by Burr, Erd\H{o}s, Graham, and S\'os \cite{BEGS} in 1980s, in perhaps the simplest interesting\footnote{The arguably more natural case of cliques is rather simple and understood very well already in \cite{BEGS}.} instance, namely that of $H$ being an odd cycle $C_{2k+1}$. 
We note that as a particularly simple instance of the Erd\H{o}s-Stone-Simonovits Theorem, we have $\ex(n,C_{2k+1})=\left\lfloor n^2/4 \right\rfloor$ for all $k\ge 1$ when $n$ is large enough.
So, the question raised by Burr, Erd\H{o}s, Graham, and S\'os \cite{BEGS} asks for the behaviour of $f\left(n,\left\lfloor n^2/4 \right\rfloor +1,C_{2k+1}\right)$.

A striking trichotomy is observed in the behavior of $f\left(n,\left\lfloor n^2/4 \right\rfloor +1,C_{2k+1}\right)$: it is constant for $C_3$, linear in $n$ for $C_5$, and quadratic in $n$ for all longer odd cycles.
It is easy to see that $f\left(n,\left\lfloor n^2/4 \right\rfloor +1,C_3\right)=3$.
For $C_5$, Erd\H{o}s and Simonovits (see \cite{BEGS}) proved $f\left(n,\left\lfloor n^2/4 \right\rfloor +1,C_5\right)=\left\lfloor n/2 \right\rfloor+3$ when $n$ is large enough.
For odd cycles $C_{2k+1}$ of length at least $7$, Burr, Erd\H{o}s, Graham, and S\'os \cite{BEGS} showed that $f\left(n,\left\lfloor n^2/4 \right\rfloor +1,C_{2k+1}\right)\ge cn^2$ for some constant $c>0$. The natural next question is to determine the correct asymptotics, namely, to determine the correct leading constant. This led Burr, Erd\H{o}s, Graham, and S\'os \cite{BEGS} to pose the following conjecture (also reiterated in an open problem collection by Erd\H{o}s \cite{E}; see Problem \#809 on Erdős Problems website~\cite{B809}).

\begin{conjecture}[\cite{BEGS,E}]\label{conj:main2}
    Let $k\ge 3$ be an integer. Then, 
    $$f\left(n,\left\lfloor \frac{n^2}{4} \right\rfloor+1,C_{2k+1}\right)=\frac{n^2}{8}+o(n^2).$$
\end{conjecture}

\noindent We prove this conjecture for all $k\ge 4$. In fact, we prove the following significantly more general result, which determines the asymptotic value of $f\left(n,e,C_{2k+1} \right)$ for the entire range of $e$.

\begin{theorem}\label{main}
    Let $k\ge 4$ be an integer. For $\left\lfloor n^2/4 \right\rfloor+1\le e\le \binom{n}{2}$, we have
    $$f\left(n,e,C_{2k+1}\right)=\frac{e}{2}+\frac{n}{2}\sqrt{e-\frac{n^2}{4}}+o(n^2).$$
\end{theorem}

\medskip

\noindent\textbf{Notation.} Throughout the paper, we use the following notation. Let $G=(V,E)$ be a graph. For a vertex $x \in V$, we denote by $N(x)$ the set of neighbors of $x$, and write $d(x):=|N(x)|$ for the degree of $x$. 
For a vertex set $A\subset V$, we denote by $G[A]$ the subgraph of $G$ induced by $A$. 
For a path $P$, the length $\ell (P)$ of $P$ is the number of edges in $P$. For any two vertices $x,y \in V$, we denote by $\di(x,y)$ the distance between $x$ and $y$ in~$G$.

\medskip

\noindent\textbf{Organization.}  
The remainder of the paper is organized as follows.  
In Section~\ref{Sec:sketch}, we provide a proof sketch of our main result, Theorem~\ref{main}.  
In Section~\ref{3}, we present several preliminary lemmas on the existence of short paths between pairs of vertices in dense graphs.  
In Section~\ref{4}, we give a complete proof of Theorem~\ref{main}.  
 
\section{Proof sketch}\label{Sec:sketch}
\noindent In this section, we outline the main ideas of our proof.
The upper bound on $f(n,e,C_{2k+1})$ follows by taking two vertex disjoint cliques of appropriate sizes, and coloring all edges of the same clique using different colors (reusing colors between the two cliques). This example was already observed in \cite{BEGS} and motivated Conjecture~\ref{conj:main2}. 

Our main contribution is the proof of the matching lower bound.
We will show that for any constant $\varepsilon>0$, 
    \begin{align}\label{sketch}
        f(n,e,C_{2k+1})\ge \frac{e}{2}+\frac{n}{2}\sqrt{e-\frac{n^2}{4}}-\varepsilon n^2,
    \end{align}
holds for any $\frac{n^2}4 < e \le \binom{n}{2},$ with $n$ large enough.

The proof proceeds by induction on $n$, namely we will assume that inequality \eqref{sketch} holds for any graph on $n-1$ vertices. Let now $G$ be an $f(n,e,C_{2k+1})$-edge-colored graph with all $C_{2k+1}$ being rainbow.

Since the $(n-1)$-vertex graph obtained from $G$ by deleting a vertex of minimum degree $\delta$ still has all $C_{2k+1}$ being rainbow, we know that it uses at least $f(n-1,e-\delta, C_{2k+1})$ colors. So, $f(n,e, C_{2k+1}) \ge f(n-1,e-\delta, C_{2k+1})$. 
Using the inductive assumption, \eqref{sketch} provides a lower bound on $f(n-1,e-\delta, C_{2k+1})$. 
This same lower bound already suffices to show \eqref{sketch} holds for $f(n,e, C_{2k+1})$ unless the minimum degree $\delta$ is relatively large. 

Therefore, it suffices to consider the situation when $\delta$ is relatively large (see \eqref{g} for a precise description). 
The proof then splits into two cases depending on whether the number of edges is substantially larger than $n^2/4$, that is, whether 
$e=e(G) \ge \left(\frac{1}{4} + \varepsilon^6 \right)n^2.$

In the case when $e\ge \left(\frac{1}{4} + \varepsilon^6 \right)n^2$, relying in addition on the lower bound on $\delta$ mentioned above, we can apply two structural lemmas which we prove in Section~\ref{3} to conclude:
\begin{itemize}
    \item[1.] There exists a large vertex set $A$ in which any two vertices are at distance at most 3 in $G$ (Lemma~\ref{3.1}).
    \item[2.] Any two vertices in $G$ are connected by a path of length 4 (Lemma~\ref{del}).
\end{itemize} 
Using these lemmas, we show that any two edges incident to $A$ must belong to a common $C_{2k+1}$.
The idea is to connect the two edges by a short path (of length at most 3 given by the first lemma),
extend it to a path of length $2k-3$ (using a greedy argument), and finally use Lemma~\ref{del} (applied to the graph obtained by deleting the inner vertices of our path of length $2k-3$) to close it into a cycle of length $2k+1$. 
Consequently, since all $C_{2k+1}$ are rainbow, all edges incident to $A$ must receive distinct colors, and a counting argument yields the required lower bound.

In the case when $e< \left(\frac{1}{4} + \varepsilon^6 \right)n^2$, the above-mentioned lower bound on $\delta$ will be close to $n/2$ (see \eqref{del2} for the precise bound).
We first show that any two vertices of $G$ are connected by a path of length 2 or 3, and moreover, this path can be chosen to avoid any small set of vertices (see Claim~\ref{claim1}).
Using a well-known classical result (see Lemma~\ref{book}) on the size of a ``book'' that we can find in any graph with more edges than the corresponding bipartite Tur\'an graph, we can find an edge $pq$ with $|N(p)\cap N(q)|\ge n/6$. 
We then prove that any two edges incident to $N(p)$ lie on a common $C_{2k+1}$, and therefore must have different colors.
The construction of such a $C_{2k+1}$ is split into several cases depending on the interaction of the two edges between themselves and with the vertices $p,q$. The basic idea is similar to that in the above case; we use a greedy argument to find a long path and Claim~\ref{claim1} to close it into a cycle. Since Claim~\ref{claim1} provides connecting paths of length $\ell \in \{2,3\}$, an additional idea here is to embed an ``adjuster'' into our cycle, namely a triangle $pqr$ (where we can choose a suitable $r$ since $|N(p)\cap N(q)|\ge n/6$) which gives us two options of routing the cycle (via the single edge $pr$ or the path $pqr$ of length two). 
Finally, using the lower bound on $\delta$, once again a counting argument shows that the number of such edges matches the required lower bound.

\section{Preliminary short path connectivity lemmas}\label{3}
\noindent In this section, we introduce several preliminary lemmas enabling us to connect vertices in our (rather dense) graphs with short paths. 

The first of our lemmas shows that any graph with density at least roughly $1/2$ contains a large set of vertices, all of which are very close to each other. 
Finding large subgraphs with small diameter (or even decomposing a graph into a few such subgraphs) is a rather classical question, see \cite[Chapter 4.3]{bela} and \cite{fox-sudakov}, for just some examples. 

\begin{lemma}\label{3.1}
    Let $G$ be a graph with $n$ vertices and $e\ge \frac{n^2}{4}-\frac{n}{2}$ edges. Then, $G$ contains a vertex set $A$ such that any two vertices in $A$ are at distance at most 3 in $G$, and
    $$|A|\ge \frac{n}{2}+\sqrt{e-\frac{n^2}{4}+\frac{n}{2}}.$$
\end{lemma}

\begin{proof}
    For simplicity, set $C_1 :=\frac{n}{2}+\sqrt{e-\frac{n^2}{4}+\frac{n}{2}}$ and $C_2 :=n-C_1$.
    Note that both $C_i$ for $i=1,2$ are solutions to the quadratic equation 
    \begin{equation}\label{eq:quadratic}
        x^2-nx+\binom{n}{2}-e=0.
    \end{equation}
    
    Let $\Delta$ be the maximum degree of $G=G(V,E)$, and let $v_0$ be a vertex of degree $\Delta$. Note that the distance between any two vertices in $\{ v_0\}\cup N(v_0)$ is at most 2. If $\Delta\ge C_1-1$, then $\{ v_0\}\cup N(v_0)$ is a desired vertex set. Now, we assume that $\Delta< C_1-1$. Set
    $$X:= \max \{ d(x)\ |\ x\in V,\ \exists\ y\in V,\ d(x)\le d(y),\ \di(x,y)>3 \},$$
    or in words, $X$ is the maximum degree of a vertex for which there exists a vertex at distance more than $3$ to it with at least as large a degree.
    
    \noindent Note that if any two vertices in $G$ have distance at most 3, then we may set $A=V$ and the proof is done. So, we may assume two vertices at distance more than three exist and, in particular, that $X$ is well-defined. 
    
    We will use $X$ as a sort of threshold for whether a vertex is of high or low degree. With this in mind let $A:=\{v\in V\ |\ d(v)>X\}$ and $B:=V \setminus A=\{v\in V\ |\ d(v)\le X\}$. By the definition of $X$, the distance between any two vertices in $A$ is at most 3, and our goal is to show that $|A|\ge C_1$. Let us hence assume, towards a contradiction, that $|A|< C_1$, and hence $|B|=n-|A|> C_2$. 

    \begin{claim}\label{cla:one}
    $X> n-\Delta -2$.
    \end{claim}
    \begin{proof}
        Assume, towards a contradiction that $X\le n-\Delta -2$. Then,
    \begin{align*}
        2e=\sum_{v\in A}d(v)+\sum_{v\in B}d(v)\le |A|\Delta + |B|X\le |A|\Delta + |B|(n-\Delta -2)= |A|(2\Delta+2-n)  + n(n-\Delta -2).
    \end{align*}
    Since $\Delta\ge \frac{2e}{n}\ge \frac{n}{2}-1$, we have $2\Delta+2-n\ge 0$. Recall that $|A|< C_1$. We obtain
    \begin{align}\label{Delta}
        2e\le C_1(2\Delta+2-n)  + n(n-\Delta -2)= \Delta(2C_1-n) +(n-2)(n-C_1).
    \end{align}
    Recall that $\Delta< C_1-1$, so we have $C_1> \Delta+1\ge n/2$. Then, 
    \begin{align*}
        \Delta\overset{\eqref{Delta}}{\ge} \frac{2e-(n-2)(n-C_1)}{2C_1-n}\overset{\eqref{eq:quadratic}}{=}C_1-1.
    \end{align*}
    This is a contradiction.
    \end{proof}

    By the definition of $X$, there exist two vertices $x$ and $y$ with $d(y)\ge d(x)=X$ and $\di(x,y)>3$. 
    Note that $\{x\}\cup N(x)$ and $\{y\}\cup N(y)$ are two disjoint subsets and there is no edges between them.
    Set $D:=\{x,y\}\cup N(x)\cup N(y)$, $S:=V\backslash D$, and $\Delta_D:=\max\{d(v)\ |\ v\in D \}$. Then,
    \begin{align*}
        2e&=\sum_{v\in B}d(v)+\sum_{v\in S\backslash B}d(v)+\sum_{v\in D\backslash B}d(v)\\
        &\le |B|X + (|S|-|S\cap B|)\Delta + (|D|-|D\cap B|)\Delta_D\\
        &= |B|X +(|S|-|S\cap B|)\Delta + (n-|S|-|B|+|S\cap B|)\Delta_D\\
        &= |B|X+|S|\Delta+(n-|S|-|B|)\Delta_D-|S\cap B|(\Delta-\Delta_D).
    \end{align*}
    By the definition of $\Delta$ and $\Delta_D$, we have $\Delta\ge\Delta_D$. Thus
    \begin{align*}
        2e\le |B|X+|S|\Delta+(n-|S|-|B|)\Delta_D= |B|(X-\Delta_D)+|S|\Delta+(n-|S|)\Delta_D.
    \end{align*}
    Since $x\in D$, we have $\Delta_D\ge d(x)=X$. Recall that $|B|>C_2$. We derive
    \begin{align}\label{Deld}
        2e\le C_2(X-\Delta_D)+|S|\Delta+(n-|S|)\Delta_D=C_2X +|S|\Delta+ (n-|S|-C_2)\Delta_D.
    \end{align}
    
    We next show that $\Delta_D\le n-X-2$. First, $d(y)=|N(y)|= n-|S\cup N(x)\cup \{x,y\}|\le n-X-2$. Since $d(x)\le d(y)$, we also have $d(x)\le d(y)\le n-X-2$. 
    For any vertex $v\in N(y)$, since $v$ has no neighbors in $N(x)\cup \{x,v\}$, we have $d(v)\le n-X-2$. Similarly, for any vertex $v\in N(x)$, we have $d(v)\le n-d(y)-2\le n-X-2$. Therefore, we have $\Delta_D\le n-X-2$.

    Since $X> n-\Delta -2$ (by Claim~\ref{cla:one}), and $\Delta<C_1-1$, we have 
    $$n-|S|-C_2\ge |N(x)\cup \{x\}|-C_2\ge X+1-C_2> n-\Delta -1-C_2>n-C_1-C_2=0.$$
    Thus, inequality \eqref{Deld} yields
    \begin{align*}
        2e\le C_2X +|S|\Delta+ (n-|S|-C_2)(n-X-2)
        = |S|(\Delta-n+X+2)+n(n-X-2)-C_2(n-2X-2).
    \end{align*}
    Since $X> n-\Delta -2$ (by Claim~\ref{cla:one}), and $|S|=n-|N(y)|-|N(x)|-2\le n-2X-2$, we further deduce that
    \begin{align}\notag
        2e&\le (n-2X-2)(\Delta-n+X+2)+n(n-X-2)-C_2(n-2X-2)\\ \label{n2x}
        &= (\Delta-C_2)(n-2X-2)+2(X+1)(n-X-2).
    \end{align}

    We next show that $n-2X-2> 0$. Otherwise, assume $n-2X-2\le 0$. Then, $X\ge \frac{n}{2}-1$, and hence $d(y)\ge d(x)=X\ge \frac{n}{2}-1$. 
    Note that there are no edges between $N(x)\cup \{x\}$ and $N(y)\cup \{y\}$. 
    We then can deduce that
    $$e\le \binom{n}{2}-(X+1)(d(y)+1)\le \binom{n}{2}-\frac{n}{2}\cdot\frac{n}{2}=\frac{n^2}{4}-\frac{n}{2}.$$
    By our initial assumption that $e\ge \frac{n^2}{4}-\frac{n}{2}$, we have $e= \frac{n^2}{4}-\frac{n}{2}$, so $C_1=\frac{n}{2}\le X+1 =d(x)+1 \le \Delta +1 $, which is a contradiction. 
    Thus, $n-2X-2> 0$.
    
    Now, according to inequality \eqref{n2x}, we have
    \begin{align*}
        \Delta&\ge \frac{2e-2(X+1)(n-X-2)}{n-2X-2}+C_2\\
        & = \frac{2e+2X^2-2nX+4X-n+2}{n-2X-2}+C_2-1\\
        & = \frac{(C_1-1-X)^2+(C_2-1-X)^2}{C_1-1-X + C_2-1-X}+C_2-1\\ &\ge (C_1-1-X)-(C_2-1-X)+C_2-1=C_1-1,
    \end{align*}
    where in the second equality we used $2e=n^2-n-2C_1C_2$ and $C_1+C_2=n$ (coming from \eqref{eq:quadratic}), and in the final inequality we used that $C_1-1-X + C_2-1-X=n-2X-2>0$. 
    This is a contradiction to $\Delta < C_1-1$, completing the proof of Lemma~\ref{3.1}.
\end{proof}

The next lemma provides a handy, sufficient condition that guarantees the existence of a path of length 4 between any two vertices.

\begin{lemma}\label{del}
    Let $G$ be a graph with $n$ vertices and $e$ edges such that $e\ge \frac{n^2}{4}+4$. If the minimum degree $\delta$ of $G$ satisfies
    \begin{align*}
        \delta\ge \frac{n}{2}-\sqrt{e-\frac{n^2}{4}}+2,
    \end{align*}
    then any two vertices in $G$ are connected by a path of length exactly 4.
\end{lemma}

\begin{proof}
    Assume, for the sake of contradiction, that there exist two vertices $x$ and $y$ that are not connected by a path of length $4$. Without loss of generality, assume that $d(x)\ge d(y)$. 

    If $x$ and $y$ are not adjacent, add an edge between $x$ and $y$ to obtain a new graph $G'$. Note that adding an edge does not decrease the minimum degree, and it increases the number of edges, so $G'$ still satisfies the assumptions of the lemma. If there is a path of length 4 connecting $x$ and $y$ in $G'$, then the same path exists in $G$. Thus, it suffices to consider the case when $x$ and $y$ are adjacent.

    Now, assume that $x$ and $y$ are adjacent. Set $A:=N(x)\cap N(y)$, $X:=N(x)\backslash (A\cup \{y\})$, $Y:=N(y)\backslash (A\cup \{x\})$, and $S:=V\backslash (N(x)\cup N(y))$. Since there are no paths of length 4 connecting $x$ and $y$, we conclude that
    
    \begin{itemize}
        \item[\textbf{(P)}]\label{itm:P} any two distinct vertices $a,b$ with $a\in A\cup X$ and $b\in A\cup Y$, have no common neighbors in $V\backslash \{x,y\}$. 
    \end{itemize}

    We proceed to show that, under our minimum degree assumptions from the lemma this property implies there are too few edges. We will do that by upper bounding the sum of the degrees of vertices in each of the sets $S,A,X,$ and $Y$. 
    
    We start with $S$. Note that by property \textbf{(P)}, any vertex $z\in S$ can have at most one neighbor in $A \cup X$ or at most one neighbor in $A \cup Y$. This implies that $d(z) \le n-3-(\min(|A \cup X|,|A \cup Y|)-1)=n-1-d(y).$  Summing up these inequalities for every $z \in S$, we get   
    \begin{align}\label{S}
        \sum_{z\in S}d(z)\le (n-d(x)-d(y)+|A|)(n-d(y)-1).
    \end{align}

    We now turn to bounding the degree of vertices in $A$. Let us first denote, in order to simplify notation, by $\delta_1:=\frac{n}{2}-\sqrt{e-\frac{n^2}{4}}+2$ and by $\Delta_1:=n-\delta_1$. Note that the lemma assumption guarantees $\delta \ge \delta_1$. We claim~that 
    \begin{align}\label{A}
        \sum_{z\in A}d(z) \le (|A|-1)(n-\Delta_1-1)+\Delta_1+3.
    \end{align}
    Indeed, if $|A| = 1$ the unique vertex in $A$ can not have neighbors in both $X$ and $Y$ by property \textbf{(P)} so its degree is upper bounded by $(n-1)-|Y|= n-d(y)+1\le n-\delta_1+1=\Delta_1+1$, as desired. On the other hand, if $|A| \ge 2$, we get
    \begin{align}\label{eq:A2}
        \sum_{z\in A}d(z)=\sum_{v\in V} |N(v) \cap A| \le n-2+2|A| \le (|A|-1)(n-\Delta_1-1)+\Delta_1+3.
    \end{align}
    where in the first inequality we used that $|N(v) \cap A| \le 1$, for any $v \in V\backslash \{x,y\}$, which, once again, holds by property \textbf{(P)}, and for the second inequality, which is equivalent to $0\le (|A|-2)(n-\Delta_1-3)$, we used that $|A| \ge 2$, and that $ \delta_1 >2 \implies \Delta_1 = n-\delta_1 \le n-3$.

    Turning now to bounding the degrees of vertices in $X$ and $Y$, let $\Delta:=\max \{d(z)|\ z\in X \}$, where we set $\Delta:=0$ if $X=\emptyset$. Assuming $X \neq \emptyset$, let $z_0\in X$ be a vertex with $d(z_0)=\Delta$. Then,
    \begin{align}\label{X}
        \sum_{z\in X}d(z)\le |X| \Delta =(d(x)-|A|-1)\Delta.
    \end{align}
    For any vertex $z\in Y \cup (A \setminus N(z_0))$, we have $d(z)\le n-\Delta$. Indeed, if $X=\emptyset$, this is trivial (as $\Delta=0$), and otherwise it holds since $z$ and $z_0$ can have no common neighbors, by property \textbf{(P)}, except possibly $x$ if $z \in A \setminus N(z_0)$, in which case $d(z) \le |(V \setminus \{z_0\}) \setminus (N(z_0) \setminus \{x\})| \le n-1-(\Delta-1)=n-\Delta$. Since also $d(z)\ge\delta_1$, this implies $\Delta\le n-\delta_1=\Delta_1$, where we know at least one such $z$ exists since $|A \cup Y|=d(y)-1 > 1$, and $|A \cap N(z_0)| \le 1$ by property (\textbf{P}). Moreover, this inequality gives,
    \begin{align}\label{Y}
        \sum_{z\in Y}d(z) \le |Y|(n-\Delta)= (d(y)-|A|-1)(n-\Delta).
    \end{align}
    Summing up \eqref{X}, and \eqref{Y}, and using that $\Delta\le\Delta_1$ and $d(x)\ge d(y)$, we get
    \begin{align}\label{X+Y}
        \sum_{z\in X \cup Y}d(z) \le (d(x)-d(y))\Delta+(d(y)-|A|-1)n \le (d(x)-d(y))\Delta_1+(d(y)-|A|-1)n.
    \end{align}
    Using the handshake lemma and summing \eqref{S}, and \eqref{X+Y},  we derive
    \begin{align*}
       2e \le\ &d(x)+d(y)+(n-d(x)-d(y)+|A|)(n-d(y)-1) +(d(x)-d(y))\Delta_1 + (d(y)-|A|-1)n + \sum_{z\in A}d(z)
    \end{align*}
    Since $d(x)$ appears on the right-hand side of this expression with the coefficient $1-n+d(y)+1+\Delta_1\ge \delta_1+\Delta_1-n+2=2>0$, and $d(x)\le n+|A|-d(y)$, we can replace it by $n+|A|-d(y)$ and obtain 
    \begin{align*}
        2e &\le n+|A|+(n+|A|-2d(y))\Delta_1 + (d(y)-|A|-1)n+\sum_{z\in A}d(z)\\
        &=d(y)(n-2\Delta_1)+n\Delta_1-|A|(n-\Delta_1-1) + \sum_{z\in A}d(z)\\
        &\overset{\eqref{A}}{\le}  d(y)(n-2\Delta_1)+n\Delta_1+2\Delta_1-n+4.
    \end{align*}

    \noindent Since $e\ge \frac{n^2}{4}+4 \implies \Delta_1\ge \frac{n}{2}$, using  $d(y)\ge \delta_1$ we obtain
    \begin{align*}
        2e &\le\delta_1(n-2\Delta_1) +n\Delta_1 +2\Delta_1-n+4\\
        &=n^2-2\delta_1\Delta_1+2\Delta_1-n+4\\
        &=n^2-2(\delta_1-2)(\Delta_1+2)+3(\delta_1-\Delta_1)-4\\
        &=8+2e-3\sqrt{4e-n^2},
    \end{align*}
    which contradicts the assumption that $e> \frac{n^2}{4}+4$, completing the proof of Lemma~\ref{del}.
\end{proof}

We note that this lemma is tight in terms of the distance. Indeed, if we take a complete graph on $n-2$ vertices, add two vertices $u,v$ adjacent respectively to two disjoint sets of vertices $A,B$, both of size given by the lower bound on the minimum degree in the above lemma, and then remove all the edges between the sets $A$ and $B$, we obtain a graph with diameter four and more than $e$ edges.

The following simple yet useful ``greedy'' observation will be used several times in the proof of Theorem~\ref{main}.

\begin{lemma}\label{greedy}
    Let $G(V,E)$ be a graph with minimum degree $\delta$. Let $k\ge 0$ be an integer, and let $S\subset V$ be a vertex set such that $|S|\le \delta -k$. Then, for any vertex $v\in V\backslash S$, there is a path $P$ of length $k$ starting at $v$ such that all vertices in $S$ are not in $P$.
\end{lemma}

\begin{proof}
    Let $v_0 v_1 \dots v_i$ be a longest path avoiding $S$, with $v_0=v$. If $i \ge k$, then $v_0,\ldots, v_k$ is a desired path. Otherwise, since $v_i$ has at least $\delta$ neighbors, and $|S \cup \{v_j\ |\ 0\le j\le i,\  j\ne i\}| = |S| + i < |S| + k\le \delta$, we can choose a neighbor $v_{i+1}$ of $v_i$ that is not in $S \cup \{v_j\ |\ 0\le j\le i,\  j\ne i\}$, giving us a contradiction to the maximality.
\end{proof}

We will also make use of the following classical result, proved independently by Edwards in an unpublished manuscript, and by Khad\v{z}iivanov and Nikiforov in \cite{KH}.

\begin{lemma}\label{book}
    Let $G$ be a graph with $n$ vertices and $e> n^2/4$ edges. Then, $G$ contains two adjacent vertices $p$ and $q$ with more than $n/6$ common neighbors.
\end{lemma}

\section{Rainbow cycles of odd length}\label{4}
\noindent In this section, we present the proof of Theorem \ref{main}.
\begin{proof}[Proof of Theorem \ref{main}]
    We first show the upper bound on $f(n,e,C_{2k+1})$.
    Let $G$ be the graph obtained by the disjoint union of two cliques $A,B$ of sizes $\left\lceil\frac{n}{2}+\sqrt{e+n-\frac{n^2}{4}}\right\rceil$ and $\left\lfloor\frac{n}{2}-\sqrt{e+n-\frac{n^2}{4}}\right\rfloor$, respectively. 
    Then, the number of edges of $G$ satisfies
    $$\binom{n}{2}-|A||B|\ge e.$$
    Let us color all edges in $A$ with different colors, and color all edges in $B$ with different colors chosen from those used for the edges of $A$. 
    It is clear that all cycles in $G$ are rainbow. 
    Thus, $f(n,e,C_{2k+1})$ is upper bounded by the number of colors used in $G$, namely
    \begin{align*}
        \binom{|A|}{2}\le \frac{e}{2}+\frac{n}{2}\sqrt{e-\frac{n^2}{4}}+ o(n^2).
    \end{align*}

    Next, we prove the lower bound on $f(n,e,C_{2k+1})$.
    We fix an integer $k\ge 4$ and treat it as a constant throughout the proof.
    It suffices to show that for any constant $0<\varepsilon<0.01$, 
    and any positive integers $n,e$ such that $e>n^2/4$, we have
    \begin{align}\label{f}
        f(n,e,C_{2k+1})\ge \frac{e}{2}+\frac{n}{2}\sqrt{e-\frac{n^2}{4}}-\varepsilon n^2-2k^2 \eps^{-26}=:g(n,e).
    \end{align}
    Note first that since $f(n,e,C_{2k+1}) \ge 0$, inequality \eqref{f} is trivially satisfied unless $2k^2 \eps^{-26} \le n^2/2$. In particular, we may and will assume that $n \ge 2k\eps^{-13}$. We will prove \eqref{f} by induction on $n$, with the preceding observation serving as the base case. So let us suppose that $f(n-1,e',C_{2k+1}) \ge g(n-1,e')$ holds, for some $n \ge 2k\eps^{-13}$ and every $\frac{(n-1)^2}4< e' \le \binom{n-1}{2}$, and our goal is to show $f(n,e,C_{2k+1}) \ge g(n,e)$, for every $\frac{n^2}4 < e \le \binom{n}{2}$.
    To this end, let $G=(V,E)$ be a graph with $n$ vertices and $e>n^2/4$ edges with an edge-coloring using  $f(n,e,C_{2k+1})$ colors satisfying the condition that every $C_{2k+1}$ is rainbow.
    
    Since by deleting a vertex of minimum degree $\delta$ in $G$ we obtain an $n-1$ vertex graph with $e-\delta$ edges, colored using $f(n,e,C_{2k+1})$ colors and still satisfying the condition that every $C_{2k+1}$ is rainbow, we can conclude by our inductive assumption that $f(n,e,C_{2k+1})\ge g(n-1,e-\delta).$ So, we are done unless $g(n,e)>g(n-1,e-\delta).$ 
    That is
    $$\frac{e}{2}+\frac{n}{2}\sqrt{e-\frac{n^2}{4}}-\varepsilon n^2-2k^2 \eps^{-26}> \frac{e-\delta}{2}+\frac{n-1}{2}\sqrt{e-\delta-\frac{(n-1)^2}{4}}-\varepsilon (n-1)^2-2k^2 \eps^{-26},$$
    which is equivalent to
    \begin{align}\label{g}
        \delta > (n-1)\sqrt{e-\delta -\frac{(n-1)^2}{4}}-n\sqrt{e-\frac{n^2}{4}}+4\varepsilon n-2\varepsilon.
    \end{align}

    \noindent This inequality provides a lower bound on $\delta$ of differing strength depending on whether $e$ is much larger than $n^2/4$ or only slightly larger. Since we need to be very careful with our estimates, we divide the proof into two cases according to how large $e$ is.

    \noindent{\bf \underline{Case 1.}} We have $e \ge \left(\frac{1}{4}+\varepsilon^6\right)n^2$.
    
    For simplicity, set $\varepsilon_1:=\frac{e}{n^2}-\frac{1}{4}\le \frac14$. Note that, by the case assumption, we have $\varepsilon_1\ge \varepsilon^6$. Then, inequality \eqref{g} can be transformed into 
    \begin{align}
    \delta &> 
    \label{sqrt}
         (n-1)\left(\sqrt{\varepsilon_1 n^2-\delta +\frac{2n-1}{4}}-\sqrt{\varepsilon_1 n^2}\right)-\sqrt{\varepsilon_1}n+4\varepsilon n-2\eps.
    \end{align}
    By using the inequality $\sqrt{1+x}-1 \ge \frac x2 -\frac{x^2}{4}$, which holds for any real $x \in [-1/2,1/2]$, with $x=\frac{2n-1-4\delta}{4\eps_1n^2},$ for which we indeed have $|x|\le  \frac{1}{2\eps_1 n} \le \frac12$, we get 
    
    \begin{align*}
        \sqrt{\varepsilon_1 n^2-\delta +\frac{2n-1}{4}}-\sqrt{\varepsilon_1 n^2}=\sqrt{\eps_1 n^2} \cdot (\sqrt{1+x}-1) &\ge \sqrt{\eps_1 n^2} \cdot \left(\frac x2 -\frac{x^2}{4}\right)\\
        &=\frac{2n-1-4\delta}{8\sqrt{\eps_1n^2}}-\frac{(2n-1-4\delta)^2}{64\eps_1^{3/2}n^3}
        \\
        &> \frac{1}{4\sqrt{\eps_1}}-\frac{\delta}{2\sqrt{\eps_1} \cdot n}-\frac{1}{2\eps_1^{3/2}\cdot n}.
    \end{align*}
    Substituting this into inequality \eqref{sqrt}, we derive
    \begin{align*}
        \delta> \frac{n-1}{4\sqrt{\varepsilon_1}} -\frac{\delta}{2\sqrt{\varepsilon_1}} - \frac{1}{2\eps_1^{3/2}}-\sqrt{\varepsilon_1}n+4\varepsilon n-2\eps > \frac{n}{4\sqrt{\varepsilon_1}} -\frac{\delta}{2\sqrt{\varepsilon_1}} -\sqrt{\varepsilon_1}n+4\varepsilon n-\frac{1}{\eps_1^{3/2}}.
    \end{align*}
    It follows that
    \begin{align*}
        \delta >\frac{\frac{1}{2}-2\varepsilon_1+8\varepsilon\sqrt{\varepsilon_1}}{1+2\sqrt{\varepsilon_1}}n - 2\eps_1^{-3/2}  
        = \left( \frac{1}{2}-\sqrt{\varepsilon_1} \right)n +\frac{8\varepsilon\sqrt{\varepsilon_1}}{1+2\sqrt{\varepsilon_1}}n-2\eps_1^{-3/2}.
    \end{align*}
    Since $\varepsilon^6\le\varepsilon_1\le \frac{1}{4}$, we obtain $\frac{8\varepsilon\sqrt{\varepsilon_1}}{1+2\sqrt{\varepsilon_1}}\geq \frac{8\varepsilon^4}{1+2\cdot (1/2)}=4\varepsilon^4$, and combined with $\eps^4 n \ge \eps_1^{-3/2}$ we get
    \begin{align}\label{case1del}
        \delta &> \frac{n}{2}-\sqrt{e-\frac{n^2}{4}} +2\varepsilon^4 n \notag \\ &\ge \frac{n}{2}-\sqrt{e-2kn-\frac{n^2}{4}}-\sqrt{2kn}+2\eps^4 n \notag \\ &\ge \frac{n}{2}-\sqrt{e-2kn-\frac{n^2}{4}}+2k,
    \end{align}
    where in the second inequality we used $e-n^2/4 \ge \eps^6 n^2 \ge 2kn$, and in the third $n \ge 2 k \eps^{-8} \implies 2\eps^4 n \ge \sqrt{2kn}+2k$. This gives us the desired lower bound on $\delta$ for this case.
    
    \noindent Finally, by our case assumption we also have
    \begin{align}\label{gar}
        e\ge \frac{n^2}{4} +\eps^6 n^2 > \frac{n^2}{4}+2kn.
    \end{align}

    \noindent Inequalities \eqref{case1del} and \eqref{gar} allow us to apply Lemma~\ref{del} to any induced subgraph $G'$ of $G$ obtained by deleting $2k-4$ of its vertices. Indeed, let $n',e',\delta'$ be the number of vertices in $G'$, the number of edges in $G'$, and the minimum degree of $G'$, respectively.
    Then, we have $n'=n-2k+4$, $e'\ge e-(2k-4)n$, and $\delta'\ge \delta-(2k-4)$.
    So,
    \begin{align*}
        e'\ge e-(2k-4)n \overset{\eqref{gar}}{>} \frac{n^2}{4}+2kn-(2k-4)n\ge \frac{(n')^2}{4}+4.
    \end{align*}
    Furthermore, we have 
    \begin{align*}
        \delta'\ge \delta - (2k-4)\overset{\eqref{case1del}}{\ge} \frac{n}{2}-\sqrt{e-2kn-\frac{n^2}{4}}+2k-(2k-4)
        >\frac{n'}{2}-\sqrt{e'-\frac{(n')^2}{4}}+2,
    \end{align*}
    where in the final inequality we used that $e'-\frac{(n')^2}{4}\geq e-2kn-\frac{n^2}{4}$.
So the conditions of Lemma~\ref{del} are satisfied for $G'$. Since $G'$ was arbitrary, we conclude that any two vertices of $G$ can be joined by a path of length exactly $4$ avoiding an arbitrary set of up to $2k-4$ vertices. 

According to Lemma~\ref{3.1}, $G$ contains a vertex set $A$ such that the distance between any two vertices in $A$ is at most 3, and
    \begin{align}\label{Alower}
        |A|\ge \frac{n}{2}+\sqrt{e-\frac{n^2}{4}+\frac{n}{2}}>\frac{n}{2}+\sqrt{e-\frac{n^2}{4}}.
    \end{align}
    We say an edge is {\bf good} if at least one of its endpoints lies in $A$. 
    To conclude the argument we now show that any two good edges lie on a common cycle of length $2k+1$ in $G$, and hence have distinct colors.
    
    Let $xy$ and $zw$ be two distinct good edges. Let $P_1$ be a shortest path (of length possibly zero) connecting an endpoint of $xy$ and an endpoint of $zw$. Without loss of generality, assume that the endpoints of $P_1$ are $x$ and $z$. Since $xy$ and $zw$ are both good edges, there exist an endpoint of $xy$ and an endpoint of $zw$ that both lie in $A$ and hence have distance at most $3$. Therefore, we have $\ell(P_1)\le 3$.
    Since, by inequality \eqref{case1del}, $\delta \ge 2k$, by Lemma~\ref{greedy} there exists a path $P_2$ of length $2k-5-\ell(P_1)\ge 0$ connecting $y$ and some vertex $u$ such that $P_2$ and the path $wzP_1x$ are vertex-disjoint. 
    As argued above, we can find a path $P_3$ of length exactly $4$ joining $w$ and $u$, while avoiding the $2k-4$ inner vertices of the path $wzP_1xyP_2u$.
    Then, $wzP_1xyP_2uP_3w$ is a cycle of length exactly $2k+1$ containing both edges $xy$ and $zw$.

    \begin{figure}[h]
\centering
\captionsetup{width=0.7\textwidth}
\includegraphics[width=0.35\linewidth]{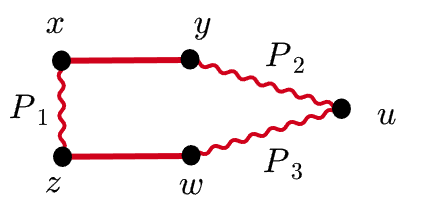}
    \caption{Illustration of how a cycle of length $2k+1$ containing both good edges $xy$ and $zw$ is constructed. Straight lines depict edges, curvy ones denote paths. We have $0\le \ell(P_1) \le 3, \:\ell(P_2)=2k-5-\ell(P_1),$ and $\ell(P_3)=4$.}
\label{fig:F1}
\end{figure}

    Therefore, any two good edges lie on a common cycle of length $2k+1$, and hence have distinct colors.
    By inequality \eqref{Alower}, the number of good edges, so also the number of distinct colors used, is at least
    \begin{align*}
        e-\binom{n-|A|}{2}\ge e- \frac{1}{2}(n-|A|)^2
        > e- \frac{1}{2}\left(\frac{n}{2}-\sqrt{e-\frac{n^2}{4}}\right)^2
        =\frac{e}{2}+\frac{n}{2}\sqrt{e-\frac{n^2}{4}}
        \ge g(n,e),
    \end{align*}
    completing the induction step in this case.

    \noindent{\bf \underline{Case 2.}} We have $e< \left(\frac{1}{4}+\varepsilon^6\right)n^2$.

    Note that inequality \eqref{g} implies
    \begin{align}
        0<\delta+n\sqrt{e-\frac{n^2}{4}} -(n-1)\sqrt{e-\delta -\frac{(n-1)^2}{4}}=:h(n,e,\delta).
    \end{align}
    Note that $h$ is increasing with respect to $\delta$, and $h\left(n,e,\frac{n}{2}-\sqrt{e-\frac{n^2}{4}}-\frac{1}{2}\right)=0$. We derive
    \begin{align}\label{del2}
        \delta> \frac{n}{2}-\sqrt{e-\frac{n^2}{4}}-\frac{1}{2}>\frac{n}{2}-\varepsilon^3n-\frac{1}{2}.
    \end{align}

    Since we have this stronger lower bound on the minimum degree compared to the previous case, we are able to get the following better variant of Lemma~\ref{del}.

    \begin{claim}\label{claim1}
        Let $S\subset V$ be a vertex set with $|S|\le 5k$. Then, either any two vertices $x$ and $y$ in $V\backslash S$ are connected by a path $P$ of length 2 or 3 such that no vertex in $S$ lies on $P$, or the number of colors used on $G$ is at least $g(n,e)$. 
    \end{claim}

    \begin{proof}
        Let $G_1:=G[V\backslash S]$. Assume, that there exist two vertices $x$ and $y$ in $G_1$ which are not connected by a path of length 2 or 3 in $G_1$. Set $X:=N(x)\backslash (S\cup \{y\})$ and $Y:=N(y)\backslash (S\cup\{x\})$. Then, $X\cap Y=\emptyset$ and there are no edges between $X$ and $Y$. 
        Without loss of generality, assume that $|X|\ge |Y|$. Then, we have $|Y|\le n/2$.
        By inequality \eqref{del2}, and $|S|\le 5k$, we have $|X|\ge |Y|\ge d(y)-|S|-1\ge \frac{n}{2}-\varepsilon^3n-\frac{3}{2}-5k$. Then, for any vertex $z\in Y$, since $z$ has no neighbors in $N(x)\cup \{x\}$, the number of neighbors of $z$ in $Y$ is at least
        \begin{align}\notag
            \delta-|V\backslash(X\cup Y\cup \{x\})|
            &\ge \frac{n}{2}-\varepsilon^3n-\frac{1}{2}-n+1+|X|+|Y|\\ \notag
            &\ge\frac{n}{2}-\varepsilon^3n-\frac{1}{2}-n+1+2\left(\frac{n}{2}-\varepsilon^3n-\frac{3}{2}-5k\right)\\ \label{12}
            &= \frac{n}{2}-3\varepsilon^3n-10k-\frac{5}{2}\\ 
            \label{cap}
            &\ge \frac{n}{4}+5k\ge \frac{|Y|}{2}+5k.
        \end{align}
        Let $G_2:=G[Y]$. So, we have just shown that the minimum degree of $G_2$ is at least $\frac{|Y|}{2}+5k$.
        Consequently, any two vertices in $Y$ have at least $10k$ common neighbors in $Y$. We next show that any two edges in $G_2$ lie on a common cycle of length $2k+1$.

        Let $pq$ and $zw$ be any two non-adjacent edges in $G_2$. 
        Then, $p$ and $z$ have at least $10k$ common neighbors in $Y$. Let $a$ be a common neighbor of $p$ and $z$ in $Y\backslash \{q,w\}$. 
        By Lemma~\ref{greedy} and inequality \eqref{cap}, there exists a path $P$ in $G_2$ of length $2k-5$ connecting $w$ and a vertex $u$ such that $q,p,a,z$ are all not in $P$. Since $u$ and $q$ have at least $10k$ common neighbors in $Y$. There exists a common neighbor $b$ of $u$ and $q$ such that $b\notin\{p,a,z,w\}$ and $b$ is not on $P$. Then, $pq$ and $zw$ lie on the cycle $qpazwPub$ of length $2k+1$. See Figure~\ref{fig:P1}, for an illustration.

        Let $pq$ and $pz$ be two distinct edges in $G_2$. Applying Lemma \ref{greedy} yields a path $P$ in $G_2$ of length $2k-3$ connecting $z$ and a vertex $u$ such that $q,p$ are not on $P$. Similarly, there exists a common neighbor $b$ of $u$ and $q$ such that $b\notin\{p,z\}$ and $b$ does not lie on $P$. Then, $pq$ and $pz$ belong to a cycle $qpzPub$ of length $2k+1$. See Figure~\ref{fig:P2}, for an illustration.

\begin{figure}[h]
\begin{minipage}[t]{0.49\textwidth}
\centering
\captionsetup{width=\textwidth}
\includegraphics[width=0.5\linewidth]{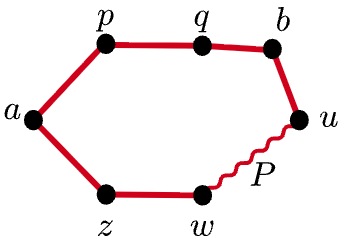}
    \caption{Illustration of how a cycle of length $2k+1$ containing both edges $pq$ and $zw$ is constructed. Straight lines depict edges, curvy one denotes a path.}
    \label{fig:P1}
\end{minipage}\hfill
\begin{minipage}[t]{0.49\textwidth}
\centering
\captionsetup{width=\textwidth}
\includegraphics[width=0.35\linewidth]{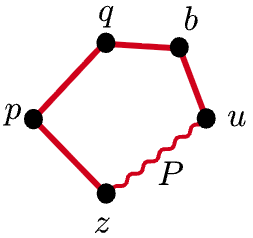}
    \caption{Illustration of how a cycle of length $2k+1$ containing both edges $pq$ and $pz$ is constructed. Straight lines depict edges, curvy one denotes a path.}
\label{fig:P2}
\end{minipage}
\end{figure}

        Therefore, any two edges in $G_2$ lie on a common cycle of length $2k+1$, and hence have different colors. By inequality \eqref{12} the number of edges in $G_2$ (and hence also the number of distinct colors used) is at least 
        \begin{align}\notag
            \frac{1}{2}|Y|\left(\frac{n}{2}-3\varepsilon^3n-10k-\frac{5}{2}\right)
            &\ge  \frac{1}{2}\left(\frac{n}{2}-\varepsilon^3n-\frac{3}{2}-5k\right)\left(\frac{n}{2}-3\varepsilon^3n-10k-\frac{5}{2}\right)\\ \notag
            &\ge \left(\frac{n^2}{8}+\frac{1}{2}\varepsilon^6 n^2\right) +\frac{1}{2}\varepsilon^3 n^2 -\varepsilon n^2\\ 
            &\ge \frac{e}{2}+\frac{n}{2}\sqrt{e-\frac{n^2}{4}}-\varepsilon n^2 > g(n,e),\label{equ:g(n,e)}
        \end{align}
        where in the second inequality we used the assumption that $\varepsilon<0.01$, and $n \ge 8k/\eps$, and in the penultimate inequality we used $e< \left(\frac{1}{4}+\varepsilon^6\right)n^2$, twice. 
        This completes the proof of Claim~\ref{claim1}.
    \end{proof}

    Since $e>n^2/4$, by Lemma \ref{book}, there is an edge $pq$ in $G$ such that $p$ and $q$ have more than $n/6$ common neighbors. Set $A:=N(p)\backslash \{q\}$. 
    In this case, we say an edge $xy$ is {\bf good} if at least one endpoint of $xy$ lies in $A$, and neither $x$ nor $y$ is in $\{p, q\}$. We next show that any two good edges belong to a common cycle of length $2k+1$. 
    As in the proof of Claim~\ref{claim1}, we split the remainder of the proof into two cases, depending on whether the two good edges are adjacent. In both cases, the assumption $k\ge 4$ is crucial.

    Let $xy$ and $xz$ be two distinct good edges (in fact, the only property of “goodness” we need here is that $\{x,y,z\}\cap \{p,q\}=\emptyset$), and let $r$ be a common neighbor of $p$ and $q$ such that $r\notin\{ x,y,z \}$. 
    By Claim \ref{claim1}, there is a path $P_1$ of length 2 or 3 connecting $p$ and $z$ such that $x,y,q,r$ are all not in $P_1$. 
    According to Lemma \ref{greedy} and inequality \eqref{del2}, there is a path $P_2$ of length $2k-5-\ell(P_1)$ connecting $y$ and a vertex $u$ such that $P_2$ and the path $rqpP_1zx$ are vertex-disjoint. 
    By Claim \ref{claim1}, there is a path $P_3$ of length 2 or 3 connecting $u$ and $r$ such that $P_3$ and the path $rqpP_1zxyP_2u$ are internally vertex-disjoint. If $\ell(P_3)=2$, then $rqpP_1zxyP_2uP_3r$ is a cycle of length $2k+1$. If $\ell(P_3)=3$, then $rpP_1zxyP_2uP_3r$ is a cycle of length $2k+1$. See Figure~\ref{fig:C1}, for an illustration.

\begin{figure}[h]
    \centering
    \begin{subfigure}[t]{0.3\textwidth}
        \centering
        \includegraphics[width=0.8\linewidth]{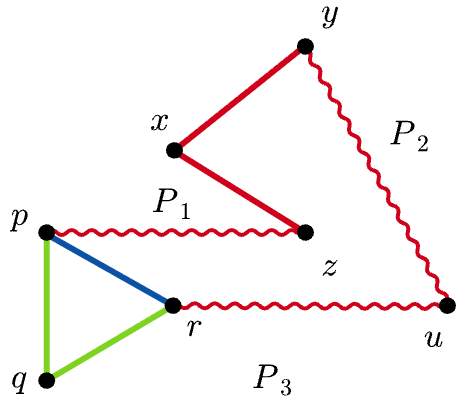}
        \caption{} 
        \label{fig:C1}
    \end{subfigure}
    \hfill
    \begin{subfigure}[t]{0.3\textwidth}
        \centering
        \includegraphics[width=0.9\linewidth]{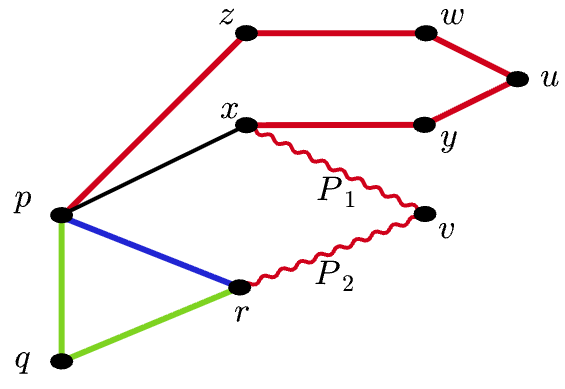}
        \caption{}
        \label{fig:C2.1}
    \end{subfigure}
    \hfill
    \begin{subfigure}[t]{0.3\textwidth}
        \centering
        \includegraphics[width=0.95\linewidth]{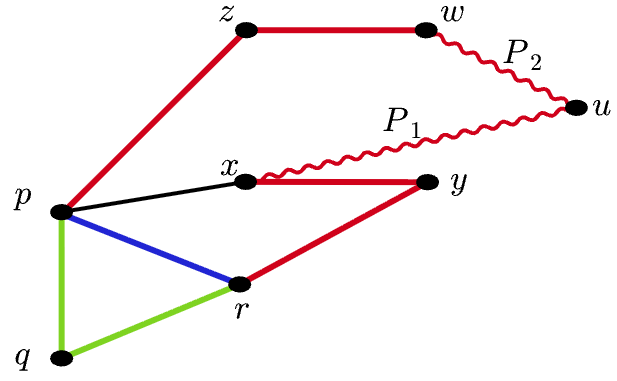}
        \caption{}
        \label{fig:C2.2}
    \end{subfigure}
    \caption{Illustration of how a cycle of length $2k+1$ containing both good edges is constructed. Straight lines depict edges, curvy ones denote paths. We use either the blue edge or the two green edges to adjust the length to precisely $2k+1$.}
    \label{fig:combined}
\end{figure}

    Let $xy$ and $zw$ be two disjoint good edges with $\{x,z\}\subset A$. 
    We divide the proof into two subcases according to whether $y$ and $w$ have a common neighbor in $V\backslash \{ p,q,x,z \}$.
    
    If $y$ and $w$ have a common neighbor $u$ such that $u\notin \{ p,q,x,z \}$, let $r$ be a common neighbor of $p$ and $q$ such that $r\notin \{ x,y,z,w,u \}$. Applying Lemma \ref{greedy} and inequality \eqref{del2}, we obtain a path $P_1$ in $G$ of length $2k-8$ connecting $x$ and a vertex $v$ such that $y,u,w,z,p,q,r$ are all not in $P_1$. 
    Applying Claim~\ref{claim1}, there exists a path $P_2$ of length 2 or 3 connecting $v$ and $r$ such that $P_2$ and the path $vP_1xyuwzpqr$ are internally vertex-disjoint. If $\ell(P_2)=2$, then $xP_1vP_2rqpzwuyx$ is a cycle of length $2k+1$. If $\ell(P_2)=3$, then $xP_1vP_2rpzwuyx$ is a cycle of length $2k+1$. See Figure~\ref{fig:C2.1}, for an illustration.
    
    Now, assume that $y$ and $w$ have no common neighbors in $V\backslash \{ p,q,x,z \}$. By inequality \eqref{del2}, we have $\min\{d(y),d(w)\}\ge \delta\ge \frac{n}{2} - 2\varepsilon^3n-\frac{1}{2}$.
    Since $|N(p)\cap N(q)|\geq \frac{n}{6}$, there is a common neighbor $r$ of $p$ and $q$ such that $r\notin \{ x,y,z,w \}$ and $r$ is adjacent to one of $y$ and $w$. Without loss of generality, assume that $r$ is adjacent to $y$. Similarly, by Lemma \ref{greedy} and inequality \eqref{del2}, there exists a path $P_1$ of length $2k-7$ connecting $x$ and a vertex $u$ such that $y,w,z,p,q,r$ are all not in $P_1$. By Claim \ref{claim1}, there is a path $P_2$ of length 2 or 3 connecting $u$ and $w$ such that $P_2$ and the path $wzpqryxP_1u$ are internally vertex-disjoint. If $\ell(P_2)=2$, then $zpqryxP_1uP_2wz$ is a cycle of length $2k+1$. If $\ell(P_2)=3$, then $zpryxP_1uP_2wz$ is a cycle of length $2k+1$. See Figure~\ref{fig:C2.2}, for an illustration.

    Therefore, any two good edges belong to a common cycle of length $2k+1$, and hence have different colors. However, the number of good edges is at least
    \begin{align*}
        \frac{1}{2}\sum_{z\in A}(d(z)-2)&\ge \frac{1}{2}|A|(\delta -2)\ge \frac{1}{2}(\delta -1)(\delta -2).
    \end{align*}
    By inequality \eqref{del2} the number of good edges (and hence also the number of distinct colors) in $G$ is at least
    \begin{align*}
        \frac{1}{2}\left(\frac{n}{2}-\varepsilon^3n-\frac{3}{2}\right)\left(\frac{n}{2}-\varepsilon^3n-\frac{5}{2}\right) 
        \ge \frac{n^2}{8}+\frac{1}{2}\varepsilon^6 n^2 +\frac{1}{2}\varepsilon^3 n^2 -\varepsilon n^2> g(n,e),
    \end{align*}
    where in the first inequality we used the assumption that $\varepsilon<0.01$, and that $n\geq 2/\eps$, and the last inequality is already shown in \eqref{equ:g(n,e)}.
    This completes the inductive step and hence the proof of Theorem~\ref{main}.
\end{proof}

We note that we crucially used the assumption that $k \ge 4$ in both cases above. Namely, it ensures that the total length of the two short paths connecting our pair of edges (provided by our lemmas, the claim, or the construction) is not already longer than $2k-1$.
For $k=3$, we do have a more involved ``stability'' argument, allowing us to bypass this issue in the second case, leaving the first case as the main bottleneck in resolving the $k=3$ case as well.

\end{document}